\documentclass[a4paper,12pt,reqno]{amsart}
\usepackage{amsfonts}
\usepackage{amsmath}
\usepackage{amssymb}
\usepackage[a4paper]{geometry}
\usepackage{mathrsfs}
\usepackage{csquotes}
\usepackage{enumitem}
\usepackage{dirtytalk}
\usepackage{comment}
\usepackage{nccmath}
\usepackage{soul}
\usepackage{breakcites}

\usepackage[colorlinks]{hyperref}
\renewcommand\eqref[1]{(\ref{#1})} 
\numberwithin{equation}{section}
\theoremstyle{plain}
\newtheorem{thm}{Theorem}[section]
\newtheorem{prop}[thm]{Proposition}
\newtheorem{cor}[thm]{Corollary}
\newtheorem{lem}[thm]{Lemma}

\theoremstyle{definition}
\newtheorem{defn}[thm]{Definition}
\newtheorem{rem}[thm]{Remark}
\newtheorem{ex}[thm]{Example}

\newcommand\norm[1]{\left\lVert#1\right\rVert}

\def\e[#1]{{\textrm{e}}^{#1}}

\begin{document}

   \title[Stability of the $L^{p}$-Poincar\'e inequality]
 {Stability of the $L^{p}$-Poincar\'e inequality for the Lebesgue measure and Gaussian probability measure with explicit geometric dependence and applications to spectral gaps}

\author[N. Yessirkegenov]{Nurgissa Yessirkegenov}
\address{
  Nurgissa Yessirkegenov:
  \endgraf
  KIMEP University, Almaty, Kazakhstan
     \endgraf
  {\it E-mail address} {\rm nurgissa.yessirkegenov@gmail.com}
  }

  \author[A. Zhangirbayev]{Amir Zhangirbayev}
\address{
  Amir Zhangirbayev:
 \endgraf
   SDU University, Kaskelen, Kazakhstan
  \endgraf
  and 
  \endgraf
  Institute of Mathematics and Mathematical Modeling, Kazakhstan
   \endgraf
  {\it E-mail address} {\rm
amir.zhangirbayev@gmail.com}
  }

\thanks{This research is funded by the Committee of Science of the Ministry of Science and Higher Education of the Republic of Kazakhstan (Grant No. AP23490970).}


     \begin{abstract} 
     In this paper, we obtain stability results for the $L^{p}$-Poincar\'e inequality for both Lebesgue measure and Gaussian probability measure (Theorem \ref{thm stability poincare} and Theorem \ref{thm stability gauss}) that involve explicit dependence on the geometry of the domain. As a byproduct, the explicit constant allows us to recover important results of Yu, Zhong \cite{yu1986lower} and Smits \cite{smits1996spectral} (Corollary \ref{cor 3}), related to the fundamental gap conjecture of the Laplacian (resolved by Andrews and Clutterbuck \cite{andrews2011proof}), thereby providing an alternative proof. Moreover, we extend this spectral gap result to the $p$-Laplacian (Corollary \ref{cor 2}). Such gap estimates for the Dirichlet $p$-Laplacian appear to be unavailable, as also observed in \cite{dai2018fundamental}. Our approach relies on properties of the first eigenfunction of the (Gaussian) $p$-Laplacian operator and weighted Poincar\'e inequalities for log-concave measures on convex domains. 
     \end{abstract}
     \maketitle

\section{Introduction}\label{sec intro}

Suppose that there is a functional inequality $A(u) \geq B(u)$ for all $u$ in some space of functions. Let $\mathcal{M}$ denote the space of all functions which achieve an equality. The basic question here is the following: Is the inequality sensitive to small changes? In the functional sense, that means if $A(u) \geq B(u)$ is very close to an equality, does that imply that $u$ is close to the set of optimizers $\mathcal{M}$? The positive answer is usually given in the form of an inequality 
\begin{align}\label{stability general intro}
A(u)-B(u)\geq cd(u, \mathcal{M}),
\end{align}
where $c>0$ and $d(u, \mathcal{M})$ is the appropriate distance function to the set of optimizers $\mathcal{M}$.

This question was first raised by Brezis and Lieb in \cite{brezis1985sobolev} for the $L^{2}$-Sobolev inequality and later resolved completely by Bianchi and Egnell in \cite{bianchi1991note}. In particular, Bianchi and Egnell obtained that for all $u\in W_{0}^{1,2}(\mathbb{R}^n)$, the following stability inequality holds:
\begin{align*}
\int_{\mathbb{R}^{n}}|\nabla u|^{2} d x-S_{n}\left(\int_{\mathbb{R}^{n}}|u|^{\frac{2 n}{n-2}} d x\right)^{\frac{n-2}{n}} \geq c_{S} \inf _{U \in E_{Sob}} \int_{\mathbb{R}^{n}}|\nabla(u-U)|^{2} d x,
\end{align*}
where $c_{S}>0$, $S_{n}$ is the sharp Sobolev constant and $E_{Sob}$ is the manifold of optimizers of the Sobolev inequality. The literature on the stability analysis of Sobolev inequality is extensive and for the interested reader we refer to \cite{bartsch2003sobolev, cianchi2006quantitative, fusco2007sharp, cianchi2009sharp, figalli2013sharp, fusco2015quantitative, figalli2019gradient, neumayer2020note, figalli2022sharp, dolbeault2025sharp, chen2025optimal}.

Following the contributions of Brezis, Lieb, Bianchi and Egnell, stability results have been investigated for various inequalities, including isoperimetric \cite{fusco2008sharp, figalli2010mass, cicalese2012selection, brasco2012sharp, figalli2018sharp, figalli2022strong}, Brunn-Minkowski \cite{figalli2009refined, figalli2017quantitative, vanhintum2021sharp, vanhintum2019sharp, figalli2023sharp, figalli2024sharp}, Gagliardo-Nirenberg \cite{ruffini2014stability, dolbeault2016stability, nguyen2019sharp, bonforte2025stability, zhang2025stability, chen2025stability}, Hardy \cite{cianchi2008hardy, machihara2013hardy, machihara2015scaling, sano2018scaling, ruzhansky2018extended, ruzhansky2018note, roychowdhury2025critical, banerjee2026sharp}, Heisenberg-Pauli-Weyl \cite{mccurdy2020quantitative, fathi2021short, cazacu2024caffarelli, huang2025sharp, do2024scale, shaimerdenov2025sharp, do2026sharp} and $L^{2}$-Poincar\'e-Wirtinger \cite{lam2026stability} with Gaussian weights. We also refer to \cite{figalli2013stability} for an interesting survey with applications to the long-time asymptotics of
evolution equations. In this work, however, we are specifically focused on the stability of the $L^{p}$-Poincar\'e inequality.

Let $1<p<\infty$ and $\Omega$ be an open bounded domain of $\mathbb{R}^n$. Then, for every $u \in W_{0}^{1,p}(\Omega)$, we have
\begin{align}\label{poincare}
\int_{\Omega}|\nabla u|^{p}dx \geq \lambda_{1}(p,\Omega)\int_{\Omega}|u|^{p}dx,
\end{align}
where $\lambda_{1}(p,\Omega)$ is the first eigenvalue of the Dirichlet $p$-Laplacian. The literature on the stability analysis of the inequality (\ref{poincare}) seems to be very limited. However, there are some works with remainder terms \cite{fleckinger2002improved, bobkov2023improved}. For example, in \cite{fleckinger2002improved}, Fleckinger-Pell{\'e} and Tak{\'a}{\v{c}} obtained an improvement of the inequality (\ref{poincare}): Assume that $\Omega$ is either an interval in $\mathbb{R}^1$ or else a bounded domain in $\mathbb{R}^n$ ($n \ge 2$) whose boundary $\partial\Omega$ is a compact connected $C^2$-manifold. Then there exists a constant $c \equiv c(p, \Omega) > 0$ such that for all $u \in W_0^{1,p}(\Omega)$
\begin{multline}\label{takac ineq}
\int_{\Omega} |\nabla u|^p dx - \lambda_1 \int_{\Omega} |u|^p dx \\\ge c \left( |u^{\parallel}|^{p-2} \int_{\Omega} |\nabla u_1|^{p-2} |\nabla u^{\perp}|^2 dx + \int_{\Omega} |\nabla u^{\perp}|^p dx \right),
\end{multline}
where
\begin{align}\label{orthogonal decom}
u=u^{\parallel}u_{1}+u^{\perp}, \quad u^{\parallel}:=\|u_1\|_{L^2(\Omega)}^{-2} \langle u, u_1 \rangle \in \mathbb{R} \quad \text{and} \quad \langle u^{\perp}, u_1 \rangle = 0
\end{align}
and $\langle\cdot,\cdot\rangle$ is the usual scalar product in $L^{2}$. It is important to note that the result (\ref{takac ineq}), in \cite{fleckinger2002improved}, was then used to obtain an existence result concerning a $(p-1)$-homogeneous problem with Dirichlet $p$-Laplacian.

Then, Bobkov and Kolonitskii \cite{bobkov2023improved} later extended (\ref{takac ineq}) to a Friedrichs inequality with a slightly modified version of the orthogonal decomposition than in (\ref{orthogonal decom}). Furthermore, leveraging the results in \cite{damascelli2004regularity, brasco2023uniqueness}, Bobkov and Kolonitskii were able to relax the regularity assumptions on $\Omega$ and confirm the conjecture proposed in \cite[Section 2.1]{pucci2004strong} concerning the relationship between the hypotheses in \cite{fleckinger2002improved}. At this point, again, we note that the stability analysis was not performed in \cite{bobkov2023improved} and is similar in structure to \cite{fleckinger2002improved} as they then also apply the corresponding result to a similar boundary value problem at resonance as an application of the improved Friedrichs inequality.

In a very recent paper \cite{lam2026stability}, Lam, Lu and Russanov obtained a certain version of the gradient stability of the Poincar\'e-Wirtinger with Gaussian probability measure in $L^{2}(\mathbb{R}^n)$. They subsequently applied this result to derive an improved Poincar\'e-Wirtinger inequality with Gaussian weights, which they then utilized to investigate the stability of the related Heisenberg-Pauli-Weyl (HPW) inequality. These results have also been generalized to monomial weights in the same paper \cite{lam2026stability}. 

The idea of connecting the HPW identity with a suitable Poincar\'e 
inequality to establish stability results for the HPW or Caffarelli-Kohn-Nirenberg inequalities has been successfully employed before (see, e.g. \cite{do2023lp, cazacu2024caffarelli, do2024scale, shaimerdenov2025sharp, do2026sharp}). In this paper, we show how such approach can also be adapted to obtain stability results for the Poincar\'e inequalities. We do this by utilizing some of the properties of the first eigenfunction of the $p$-Laplacian \cite{sakaguchi1987concavity, colesanti2025geometric} and the weighted Poincar\'e inequality for log-concave measures \cite{ferone2012remark}. As a result, we are able to derive the following stability result with explicit constant: Let $p\geq 2$, $\Omega$ be an open bounded convex set having diameter $\mathrm{diam}(\Omega)$ with smooth boundary $\partial \Omega$. Then, for every real-valued $u\in C_{0}^{\infty}(\Omega)$, we have
\begin{align}\label{stability intro}
\int_{\Omega}|\nabla u|^{p}dx-\lambda_{1}(p,\Omega)\int_{\Omega}|u|^{p}dx \geq c_{1}(p)\left(\frac{\pi_{p}}{\mathrm{diam}(\Omega)}\right)^{p}d(u, E_{Poin})^{p},
\end{align}
where $c_{1}(p)$ is an explicit constant given in (\ref{cp sharp formula}). Here, $E_{Poin}=\{cu_{1}: c\in \mathbb{R}\}$ is the manifold of optimizers (or eigenspace) of (\ref{poincare}),
\begin{align*}
\pi_{p}=2 \int_{0}^{+\infty} \frac{1}{1+\frac{1}{p-1} s^{p}} d s=2 \pi \frac{(p-1)^{1 / p}}{p(\sin (\pi / p))}
\end{align*}
and $d(u, E_{Poin})$ is the $L^{p}$ distance to the manifold of optimizers, i.e.
\begin{align*}
d(u, E_{Poin})^{p} := \inf_{c\in\mathbb{R}}\left\{\norm{u-cu_{1}}^{p}_{L^{p}(\Omega)}\right\}.
\end{align*}

The constant, in (\ref{stability intro}), is explicit with clear dependence on the diameter of the convex domain $\Omega$, which we believe is the first such result (see, Remark \ref{simple rem} for a simpler expression of the constant). In general, in most stability results, the explicit form of the constant that appears nearby the distance function (\ref{stability general intro}) is unknown. Even for the historic result of Bianchi and Egnell \cite{bianchi1991note} for the Sobolev inequality, which was obtained in 1991, no information was known for the constant $c_{S}$ (except an upper bound \cite{konig2023sharp}) until very recently \cite{dolbeault2025sharp}. There, Dolbeault, Esteban, Figalli, Frank and Loss obtained sharp stability results for the Sobolev inequality with optimal dimensional dependence (see also, \cite{chen2024stability, chen2025asymptotically, chen2026optimal}). In our case, the inequality (\ref{stability intro}) provides not only stability for the $L^{p}$-Poincar\'e inequality, but also gives an explicit geometric dependence of the stability constant due to the result of Ferone, Nitsch and Trombetti \cite{ferone2012remark}. Additionally, we show some corresponding results for the Gaussian probability measure using the same logic thanks to \cite{colesanti2025geometric}.

Moreover, after normalizing the eigenfunctions $u_{1}$ and $u_{2}$, we obtain the following spectral gap of the Dirichlet $p$-Laplacian:
\begin{align}\label{gap p intro}
\lambda_{2}(p,\Omega)-\lambda_{1}(p,\Omega) \geq c_{1}(p)\left(\frac{\pi_{p}}{\mathrm{diam}(\Omega)}\right)^{p}C(p, \Omega, u_{1}, u_{2}),
\end{align}
where
\begin{align*}
C(p, \Omega, u_{1}, u_{2}):=\inf_{c\in \mathbb{R}}\int_{\Omega}|u_{2}-cu_{1}|^{p}dx.
\end{align*}

We could not find any results related to the spectral gaps for the Dirichlet $p$-Laplacian operator. In fact, in \cite[Page 7]{dai2018fundamental}, Dai, Seto and Wei, note that \say{similar gap estimate for the $p$-Laplacian is still unknown}. We believe (\ref{gap p intro}) is the first such estimate.

Specializing to $p=2$ in (\ref{gap p intro}), we get the fundamental gap of the Dirichlet Laplacian in the form of Yu, Zhong and Smits \cite{yu1986lower, smits1996spectral}:
\begin{align*}
\lambda_{2}(2, \Omega)-\lambda_{1}(2, \Omega)\geq \frac{\pi^{2}}{\mathrm{diam}(\Omega)^{2}}.
\end{align*}
The difference between the first two eigenvalues of the Dirichlet Laplacian on a bounded open domain $\Omega$, denoted by $\lambda_{2}(2,\Omega) - \lambda_{1}(2,\Omega)$, is 
commonly known as the fundamental gap of $\Omega$. The fundamental gap problem has an interesting history and we give a brief overview of it here. 

Originally, van den Berg, in \cite{vandenberg1983condensation}, first observed that many convex sets have a spectral gap of $\frac{3\pi^{2}}{\mathrm{diam}(\Omega)^{2}}$. This was also independently observed and then conjectured to be true in any convex domain $\Omega$ by Yau \cite{yau1987nonlinear}, Ashbaugh and Benguria \cite{ashbaugh1989optimal} for the Schr\"odinger operator $-\Delta+V$. The general statement in terms of the Schr\"odinger operator has been resolved in the one-dimensional setting by Lavine \cite{lavine1994eigenvalue} (see, also, \cite{ashbaugh1989optimal, horvath2003first}). In higher dimensions, a breakthrough was made by Singer, Wong, Yau and Yau \cite{singer1985estimate}, where they obtained a lower bound of $\frac{\pi^{2}}{4\mathrm{diam}(\Omega)^{2}}$ by utilizing the gradient estimates in spirit of those from Li \cite{li1979lower}, Li and Yau \cite{li1980estimates} and the log-concavity of the first eigenfunction (established for convex potentials by Brascamp and Lieb \cite{brascamp1976extensions}). Later, Yu, Zhong \cite{yu1986lower} and \cite{smits1996spectral} improved the estimate to $\frac{\pi^{2}}{\mathrm{diam}(\Omega)^{2}}$. Under certain assumptions on the domain $\Omega$, a major step forward in proving the conjecture was made by Ba{\~n}uelos and M{\'e}ndez-Hern{\'a}ndez \cite{banuelos2000sharp}, Davis \cite{davis2001spectral} and Ba{\~n}uelos with Kr{\"o}ger \cite{banuelos2001gradient}. Finally, Andrews and Clutterbuck \cite{andrews2011proof} proved the conjecture in 2011 for any convex domain $\Omega$. A few years later, an alternative proof was given by Ni in \cite{ni2013estimates}. 
Gap estimates on other domains are addressed in \cite{cheng1997isoperimetric, oden1999spectral, seto2019sharp, he2020fundamental, dai2021fundamental, sun2025sharp}. One of which is by Sun and Wang \cite{sun2025sharp}, where they showed sharp fundamental gap estimate on convex domains in Gaussian spaces by establishing improved log-concavity properties for the Gaussian heat kernel. In a very recent paper \cite{amato2024geometric}, Amato, Bucur and Fragal{\`a} 
strengthened the result of Andrews and Clutterbuck and quantified it in terms of flatness, thereby answering an additional open problem by Yau \cite{yau1993open}. Additional works related to the fundamental gap problem of the Neumann Laplacian are present in \cite{payne1960optimal, zhong1984estimate}. We also refer to \cite{andrews2014moduli, dai2018fundamental} for comprehensive surveys on the subject.

The paper is organized as follows. In Section \ref{sec prem}, we collect the necessary notation and preliminary results, including the variational characterization of the eigenvalues of the Dirichlet $p$-Laplacian and its Gaussian analogue, the log-concavity of the first eigenfunction, the weighted $L^{p}$-Poincar\'{e} inequality for log-concave measures and the properties of the $C_{p}$-functional together with the associated Picone-type identity. In Section \ref{main results}, we state the main results of the paper: the stability of the $L^{p}$-Poincar\'{e} inequality for the Lebesgue measure (Theorem \ref{thm stability poincare}) and for the Gaussian probability measure (Theorem \ref{thm stability gauss}) along with their consequences on the fundamental gap estimates (Corollaries \ref{cor 2}, \ref{cor 3}, \ref{cor 333}). Section \ref{sec proof prop} is devoted to the proof of Proposition \ref{cp asymp}, establishing the asymptotic behavior of the sharp constant $c_{1}(p)$. Finally, Sections \ref{sec proofs lebesgue} and \ref{sec proofs gaussian} contain the proofs of the main theorems for the Lebesgue and Gaussian settings, respectively.

\section{Preliminaries}\label{sec prem}

In this section, we briefly recall the necessary notation and provide some preliminary results. Let $x\in \mathbb{R}^{n}$ be a point in the $n$-dimensional Euclidean space $\mathbb{R}^n$. Then, $|x|$ denotes the Euclidean norm of $x$ in $\mathbb{R}^n$ with $dx$ denoting the integration with respect to the standard Lebesgue measure.

Let $1<p<\infty$, $\Omega$ be a bounded domain of $\mathbb{R}^n$. Then, a non-zero function $u\in W_{0}^{1,p}(\Omega)$ is called an eigenfunction of the $p$-Laplacian $\Delta_{p}u=\text{div}\left(|\nabla u|^{p-2}\nabla u\right)$ with a corresponding eigenvalue $\lambda \in \mathbb{R}$ if it satisfies the following in the weak sense:
\begin{align}\label{p-laplacian}
\begin{cases} 
-\Delta_{p} u = \lambda|u|^{p-2} u & \text{in } \Omega, \\ 
u = 0 & \text{on } \partial \Omega.
\end{cases}
\end{align}
When $p=2$, the problem (\ref{p-laplacian}) reduces to the eigenvalue problem for the standard Laplace operator. In general, the first eigenvalue can be characterized as
\begin{align}\label{eigenvalue first}
\lambda_{1}(p, \Omega) = \min_{u \in \mathcal{S}_{p}} \int_{\Omega} |\nabla u|^{p} dx,
\end{align}
where
\begin{align*}
\mathcal{S}_{p}:=\left\{u\in W^{1,p}_{0}(\Omega):\norm{u}_{L^{p}(\Omega)}=1\right\}.
\end{align*}
We remark that $\lambda_{1}(p, \Omega)>0$ and is attained by the associated first eigenfunction $u_{1}\in W_{0}^{1,p}(\Omega)$. All first eigenfunctions are known to coincide up to modulo scaling (see, \cite{vazquez1984strong, lindqvist1990equation, belloni2002direct}) and do not change sign in $\Omega$. Therefore, it is possible to assume without the loss of generality that $u_{1}>0$ unless stated otherwise. Furthermore, $u_{1}$ is known to be bounded and in $C_{loc}^{1,\beta}(\Omega)$ for some $\beta\in (0,1)$ (for more details, see, \cite{dibenedetto1983local, otani1988existence, tolksdorf1984regularity}). 

When $p\neq 2$, one can construct several infinite sequences of variational eigenvalues by the Ljusternik-Schnirelman theory with topological indexes such as Krasnosel’skii genus $\gamma(\mathcal{A})$ of a symmetric set $\mathcal{A}\subset W_{0}^{1,p}(\Omega)$ (see, \cite{azorero1987existence}). In particular, let $\Sigma_{k}$ be the collection of all symmetric subsets $\mathcal{A}$ of $\mathcal{S}_{p}$ with $\gamma(\mathcal{A})\geq k$. Then, the numbers
\begin{align}\label{def eigen p}
\lambda_{k}(p,\Omega)=\inf_{\mathcal{A}\in \Sigma_{k}}\max_{u\in\mathcal{A}}\int_{\Omega}|\nabla u|^{p}dx
\end{align}
form a sequence of eigenvalues such that
\begin{align*}
0 < \lambda_1(p, \Omega) < \lambda_2(p, \Omega) \leq \cdots \leq \lambda_k(p, \Omega) \to +\infty \quad \text{as } k \to +\infty.
\end{align*}
When $p=2$, the eigenvalues of (\ref{def eigen p}) coincide with the standard discrete sequence of eigenvalues of the Laplacian (see, \cite{brasco2016stability, attouch2014variational}). Unlike the $p=2$ case, it is unknown whether these are all eigenvalues of the $p$-Laplacian. It is known, however, that there are no eigenvalues between $\lambda_1(p, \Omega)$ and $\lambda_2(p, \Omega)$ with $\lambda_1(p, \Omega)$ being the smallest eigenvalue (see, \cite{anane1996second}). Then, Cuesta, De Figueiredo and Gossez, in \cite{cuesta1999beginning}, derived a variational characterization of $\lambda_{2}(p,\Omega)$ for an open bounded connected domain $\Omega$ (later proved for any open domain with finite measure by Brasco and Franzina \cite{brasco2013hong}):
\begin{align}\label{eigenvalue second}
\lambda_{2}(p,\Omega) = \inf_{\gamma \in \Gamma(u_{1}, -u_{1})} \left[ \max_{u \in \gamma([0,1])} \int_{\Omega} |\nabla u|^{p} dx \right],
\end{align}
where $u_{1}$ is the first eigenfunction such that $\norm{u_{1}}_{L^{p}(\Omega)}=1$ and $\Gamma(u_{1}, -u_{1})$ is the family of all continuous maps $\gamma$ from $[0,1]$ to $\mathcal{S}_{p}$ with endpoints $u_{1}$ and $-u_{1}$. We refer also to \cite{drabek1999resonance, perera2003nontrivial, juutinen2005higher, bobkov2017some, audoux2018multiplicity, fusco2019variational, bobkov2025rayleigh, bobkov2025inverse, bobkov2026nodal} for other results in this direction. 

We finish this part by stating an important result of Sakaguchi \cite[Theorem 1]{sakaguchi1987concavity} on the log-concavity of the eigenfunctions of the Dirichlet $p$-Laplacian operator (see, also, \cite{brascamp1976extensions}). 

\begin{thm}[\cite{sakaguchi1987concavity}]\label{thm saka}
Let $\Omega$ be a bounded convex domain in $\mathbb{R}^n \ (n \ge 2)$ with smooth boundary $\partial\Omega$. Fix a number $p > 1$. Let $u \in W_0^{1,p}(\Omega)$ be a positive weak solution to the nonlinear eigenvalue problem
\begin{align*}
\begin{cases} 
-\Delta_{p} u = \lambda|u|^{p-2} u & \text{in } \Omega, \\ 
u = 0 & \text{on } \partial \Omega.
\end{cases}
\end{align*}
Then, $v = \log u$ is a concave function.
\end{thm}

On the other hand, we have the Gaussian probability space $(\mathbb{R}^n, \gamma)$ with measure $\gamma$ defined as
\begin{align*}
\gamma(\Omega)=(2\pi)^{-\frac{n}{2}}\int_{\Omega}e^{-\frac{|x|^{2}}{2}}dx
\end{align*}
for any measurable set $\Omega\subseteq \mathbb{R}^n$. In this context, we denote integration with respect to the measure $\gamma$ as $d\gamma$. This measure is explicitly defined by the density $d\gamma(x) = (2\pi)^{-\frac{n}{2}} e^{-\frac{|x|^{2}}{2}} \, dx$. Furthermore, $d\gamma_{\partial \Omega}$ represents the $(n-1)$-dimensional Hausdorff measure on the boundary $\partial \Omega$ with respect to $\gamma$. 

Now let $1<p<\infty$ and $\Omega\subseteq \mathbb{R}^n$ be a measurable set, then the meaning of the notation $L^{p}(\Omega)$, $W^{1,p}(\Omega)$ and $W_{0}^{1,p}(\Omega)$ is same as usual. In the setting of the Gaussian probability measure, we define $L^{p}(\Omega,\gamma)$ as the space of all measurable functions $u$ such that
\begin{align*}
L^{p}(\Omega, \gamma) = \left\{ u : \Omega \to \mathbb{R} : \|u\|_{p, \gamma}^{p} := \int_{\Omega} |u(x)|^{p} d \gamma < +\infty \right\}.
\end{align*}
In the same spirit, we have $W^{1,p}(\Omega, \gamma)$, $W_{0}^{1,p}(\Omega, \gamma)$ as the $\gamma$-weighted Sobolev spaces with $W_{0}^{1,p}(\Omega, \gamma)$ being the closure of $C^{\infty}_{0}(\Omega)$. We note that if $\Omega$ is a bounded set, then the density $e^{-\frac{|x|^{2}}{2}}$ has both positive upper and lower bounds in $\Omega$, thereby implying $L^{p}(\Omega, \gamma)=L^{p}(\Omega)$, $W^{1,p}(\Omega, \gamma)=W^{1,p}(\Omega)$ and so on. Otherwise, when $\Omega$ is unbounded, we have $L^{p}(\Omega)\subset L^{p}(\Omega, \gamma)$ and $W^{1,p}(\Omega)\subset W^{1,p}(\Omega, \gamma)$ (see, e.g., \cite{kilpelainen1994weighted, heinonen2018nonlinear, papageorgiou2019nonlinear}). 

In addition, we define the notion of the Gaussian divergence:
\begin{align*}
\text{div}_{\gamma}X:=e^{\frac{|x|^{2}}{2}}\text{div}\left(e^{-\frac{|x|^{2}}{2}}X\right)=\text{div}X-\langle X, x\rangle
\end{align*}
along with the corresponding product rule:
\begin{align*}
\text{div}_{\gamma}\left(fV\right)=\langle\nabla f, V \rangle+ f\text{div}_{\gamma}V. 
\end{align*}
Let $\Omega$ be a bounded domain of $\mathbb{R}^n$. Similarly as for the Euclidean $p$-Laplacian, we define the Gaussian $p$-Laplacian operator as
\begin{align*}
-\Delta_{p, \gamma} u := -\operatorname{div}(|\nabla u|^{p-2} \nabla u) + \langle x, \nabla u\rangle|\nabla u|^{p-2},
\end{align*}
which is an extension of the Ornstein-Uhlenbeck operator as for $p=2$, we recover its usual form
\begin{align*}
L_{\gamma}u:=-\Delta u + \langle x, \nabla u \rangle.
\end{align*} 
Under appropriate regularity of functions and domains, the integration by parts formula is as follows:
\begin{align*}
\int_{\Omega} v \Delta_{p, \gamma}(u) d \gamma = -\int_{\Omega}|\nabla u|^{p-2}\langle\nabla u , \nabla v\rangle d \gamma + \int_{\partial \Omega} v|\nabla u|^{p-2}\langle\nabla u , n_{x}\rangle d \gamma_{\partial \Omega},
\end{align*}
where $n_{x}$ is the outward unit normal vector at point $x\in \partial \Omega$. We call $u\in W_{0}^{1,p}(\Omega, \gamma)$ an eigenfunction of the $p$-Gaussian operator associated to the eigenvalue $\lambda\in\mathbb{R}$ if it is a weak solution to 
\begin{align}\label{gauss p-lap}
\begin{cases}
-\Delta_{p, \gamma} u = \lambda|u|^{p-2} u & \text { in } \Omega, \\
u = 0 & \text { on } \partial \Omega.
\end{cases}
\end{align}
For the existence of solutions to (\ref{gauss p-lap}) we refer to \cite[Section 6]{franceschi2024cheeger}. Here, we define its first positive eigenvalue 
\begin{align*}
\lambda_{1}(p, \Omega, \gamma) = \min_{u\in\mathcal{S}^{\gamma}_{p}} \int_{\Omega} |\nabla u|^{p} d\gamma,
\end{align*}
where
\begin{align*}
\mathcal{S}_{p}^{\gamma}:=\left\{u\in W_{0}^{1,p}(\Omega,\gamma) : \norm{u}_{L^{p}(\Omega, \gamma)}=1\right\}.
\end{align*}
Just as in the Euclidean case, without the loss of generality, one can take the first eigenfunction $u^{\gamma}_{1}$ of the Gaussian $p$-Laplacian as positive as it does not change sign in $\Omega$ and is unique up to scalar multiplication (for more details, see, \cite{du2021estimates, franceschi2024cheeger}). 

When $p=2$, in (\ref{gauss p-lap}), by standard techniques \cite{evans2010partial} we can define variational eigenvalues for the Ornstein-Uhlenbeck operator by
\begin{align}\label{def eigenvalues gauss}
  \lambda_i(2, \Omega, \gamma) = \inf_{\substack{E_i \subset S^{\gamma}_{2} \\ E_i: \, i\text{-dim space}}}
  \max_{u \in E_i} R_{\gamma}(u),
\end{align}
where the Rayleigh quotient is given by
\begin{align*}
  R_{\gamma}(u) = \int_{\Omega} |\nabla u|^2 \, d\gamma.
\end{align*}

Finally, we recall a corresponding result of Sakaguchi for the Gaussian $p$-Laplacian operator by Colesanti, Qin, Salani \cite[Theorem 1.5]{colesanti2025geometric}, some regularity properties of the Gaussian eigenfunctions \cite[Proposition 2.1]{colesanti2025geometric}, the weighted $L^{p}$-Poincar\'e inequality for the log-concave probability measures (see, \cite[Section 6.4.3]{mazya2013sobolev} and \cite{ferone2012remark}) as well as sharp bounds on the constant \cite[Theorem 1.1]{ferone2012remark}:

\begin{thm}[\text{\cite[Theorem 1.5]{colesanti2025geometric}}]\label{thm convex gauss}
Let $p > 1$, $n \geq 2$, $\Omega$ be an open, bounded and convex domain in $\mathbb{R}^n$, and $u$ be a solution of problem (\ref{gauss p-lap}), with $u > 0$ in $\Omega$. Then the function
\begin{align*}
W = \log u
\end{align*}
is concave in $\Omega$.
\end{thm}

We remark that, in Theorem \ref{thm convex gauss}, Colesanti, Qin and Salani do not require the boundary $\partial \Omega$ to be smooth compared to Theorem \ref{thm saka}. So, it might be possible to remove the smoothness condition for the Euclidean $p$-Laplacian case too.

\begin{prop}[\text{\cite[Proposition 2.1]{colesanti2025geometric}}]\label{prop regularity gaussian}
If $\Omega$ is a bounded open domain and $u$ is a weak solution of problem (\ref{gauss p-lap}), with $p > 1$, then $u \in C^{1, \alpha}_{\text{loc}}(\Omega) \cap C^2(\Omega \setminus \bar{\mathcal{C}})$ for some $\alpha \in (0,1)$, where $\mathcal{C} = \{x \in \Omega : \nabla u(x) = 0\}$. Moreover, if $\partial \Omega$ is $C^{1,\alpha}$, then $u \in C^{1,\beta}(\bar{\Omega})$, for some $\beta \in (0,1)$.
\end{prop}

\begin{thm}[\text{\cite[Theorem 1.1]{ferone2012remark}}]\label{thm log-convex stab}
Let $p>1$ and $\omega$ be a positive log-concave function on an open bounded convex set $\Omega$ having the diameter $\mathrm{diam}(\Omega)$. Then, there exists a positive constant $C_{\Omega,p,\omega}$ such that, for every Lipschitz function $u$ with $\int_{\Omega}|u|^{p-2}u\omega dx = 0$, we have
\begin{align*}
\int_{\Omega}\left|\nabla u\right|^{p}\omega dx\geq C_{\Omega, p, \omega}\inf_{t \in \mathbb{R}} \int_{\Omega}\left|u - t\right|^{p}\omega dx.
\end{align*}
Moreover, in any dimension, we have
\begin{align*}
C_{\Omega, p, \omega}\geq\left(\frac{\pi_{p}}{\mathrm{diam}(\Omega)}\right)^{p},
\end{align*}
where
\begin{align*}
\pi_{p}=2 \int_{0}^{+\infty} \frac{1}{1+\frac{1}{p-1} s^{p}} d s=2 \pi \frac{(p-1)^{1 / p}}{p(\sin (\pi / p))}.
\end{align*}
\end{thm}
For the explicit expression of $\pi_{p}$, we refer to \cite{stanoyevitch1990geometry, lindqvist1995some, stanoyevitch1993products, reichel1999sturm, binding2006basis}. In the next part of this section, we introduce the definition of the $C_{p}$-functional and its upper and lower bounds, which will be essential for our proof. 

\begin{defn}
Let $1<p<\infty$. Then, for $\xi,\eta\in\mathbb{C}^{n}$, we define
\begin{align}\label{cp formula}
C_p(\xi,\eta):=|\xi|^p-|\xi-\eta|^p-p|\xi-\eta|^{p-2}\textnormal{Re}\langle(\xi-\eta),\eta\rangle\geq0.
\end{align}
\end{defn}

\begin{lem}[\text{\cite[Step 3 of Proof of Theorem 1.2]{cazacu2024hardy}}]\label{lem1}
Let $p\geq2$. Then, for $\xi,\eta\in\mathbb{C}^n$, we have
\begin{align}\label{cp 111}
C_{p}(\xi,\eta)\geq c_1(p)|\eta|^{p},
\end{align}
where 
\begin{align*}
c_1(p)
= \inf_{(s,t)\in\mathbb{R}^2\setminus\{(0,0)\}}
\frac{\bigl[t^2 + s^2 + 2s + 1\bigr]^{\frac p2} -1-ps}
{\bigl[t^2 + s^2\bigr]^{\frac p2}}\in(0,1].
\end{align*}
Moreover, when $p=2$, $c_{1}(2)=1$.
\end{lem}

The constant, in (\ref{cp 111}), is especially interesting. In \cite{cazacu2024hardy}, the authors asked whether it is the most optimal for the given inequality. This question was recently answered by Huang and Tong \cite{huang2025lp} and we state it here.

\begin{lem}[\text{\cite[Theorem 6]{huang2025lp}}]\label{lem chinese cp}
Let $p \geq 2$. The sharp constant $c_{1}(p)$ in (\ref{cp 111}) is determined as
\begin{align}\label{cp sharp formula}
c_{1}(p)=(p-1)(1-k_{0})^{p}+p k_{0}(1-k_{0})^{p-1}+k_{0}^{p}>0,
\end{align}
where $k_{0}=\frac{r_{0}}{1+r_{0}}$ and $r_{0}$ is the solution of the equation
\begin{align}\label{poly r}
r^{p-1}-(p-1) r-(p-2)=0.
\end{align}
\end{lem}

\begin{ex}
For $p=3$, the sharp constant is given by $c_{1}(3)=2-\sqrt{2}$.
\end{ex}

\begin{rem}
Note that since $k_{0}=\frac{r_{0}}{1+r_{0}}$, we have $1-k_{0}=\frac{1}{1+r_{0}}$. Substituting this into (\ref{cp sharp formula}) gives us
\begin{align*}
c_{1}(p)=(p-1)\left(\frac{1}{1+r_{0}}\right)^{p}+p\frac{r_{0}}{1+r_{0}}\left(\frac{1}{1+r_{0}}\right)^{p-1}+\left(\frac{r_{0}}{1+r_{0}}\right)^{p}.
\end{align*}
This simplifies to
\begin{align}\label{almost simple}
c_{1}(p)=\frac{p-1+pr_{0}+r^{p}_{0}}{(1+r_{0})^{p}}.
\end{align}
From equation (\ref{poly r}), we get that
\begin{align}\label{rp eq}
r^{p}_{0}=r^{2}_{0}(p-1)+r_{0}(p-2).
\end{align}
Now we substitute (\ref{rp eq}) to (\ref{almost simple}):
\begin{align*}
c_{1}(p)&=\frac{p-1+pr_{0}+r^{2}_{0}(p-1)+r_{0}(p-2)}{(1+r_{0})^{p}}
\\&=\frac{p-1+(2p-2)r_{0}+r^{2}_{0}(p-1)}{(1+r_{0})^{p}}
\\&=(p-1)(r_{0}+1)^{2-p}.
\end{align*}
This is a much simpler expression for $c_{1}(p)$ and it is the one which we will use throughout the text.
\end{rem}

Although Lemma \ref{lem chinese cp} makes it possible to compute the sharp constant $c_{1}(p)$ via a simple polynomial equation, it is unclear what happens as $p\to \infty$, which is interesting in its own right. That is why, using Lemma \ref{lem chinese cp}, we provide lower and upper bounds for the constant $c_{1}(p)$ (for more details see Section \ref{sec proof prop}) that exactly reveal its behavior as $p\to \infty$.

\begin{prop}\label{cp asymp}
For any $p\geq 2$, we have
\begin{align}\label{lower bound cp}
\frac{1}{2^{p-2}}\leq c_{1}(p)\leq \frac{p-1}{2^{p-2}}.
\end{align}
Moreover, as $p\to \infty$, the constant $c_{1}(p) \to 0$.
\end{prop}

\begin{lem}[\text{\cite[Lemma 2.2]{CT24}}]\label{lem2}
Let $1<p<2\leq n$. Then, for $\xi,\eta\in \mathbb{C}^{n}$, we have
\begin{align*}
C_p(\xi, \eta) \geq c_2(p) \frac{|\eta|^2}{\left( |\xi| + |\xi - \eta| \right)^{2-p}},
\end{align*}
where
\begin{align*}
c_2(p) := \inf_{s^2 + t^2 > 0} \frac{\left( t^2 + s^2 + 2s + 1 \right)^{\frac{p}{2}} - 1 - ps}{\left( \sqrt{t^2 + s^2 + 2s + 1} + 1 \right)^{p-2} (t^2 + s^2)} \in \left(0,  \frac{p(p-1)}{2^{p-1}} \right].
\end{align*}
\end{lem}

\begin{lem}[\text{\cite[Lemma 2.3]{CT24}}]\label{lem3}
Let $1<p<2\leq n$. Then, for $\xi,\eta\in \mathbb{C}^{n}$, we have
\begin{align*}
C_p(\xi, \eta) \leq c_3(p) \frac{|\eta|^2}{\left( |\xi| + |\xi - \eta| \right)^{2-p}},
\end{align*}
where
\begin{align*}
c_3(p) := \sup_{s^2 + t^2 > 0} \frac{\left( t^2 + s^2 + 2s + 1 \right)^{\frac{p}{2}} - 1 - ps}{\left( \sqrt{t^2 + s^2 + 2s + 1} + 1 \right)^{p-2} (t^2 + s^2)} \in \left[ \frac{p}{2^{p-1}}, +\infty \right).
\end{align*}
\end{lem}

Lastly, we use the complex-valued version of a Picone-type identity from \cite{apseit2025sharp} that connects to the $C_{p}$-functional (\ref{cp formula}):

\begin{thm}[\cite{apseit2025sharp}]\label{thm picone}
Let $u$ be a complex-valued function on $\Omega \subset \mathbb{R}^{n}$ and $\phi$ be a nonzero, real-valued function on $\Omega \subset \mathbb{R}^{n}$. Then, we have
\begin{align*}
&C_p(\xi, \eta) = |\nabla u|^p + (p-1) \left| \frac{\nabla \phi}{\phi} u \right|^p - p \, \textnormal{Re} \left[ \left| \frac{\nabla \phi}{\phi} u \right|^{p-2} \frac{u}{\phi}\langle\nabla \phi , \nabla u\rangle \right],
\\&R_p(\xi, \eta) = |\nabla u|^p - |\nabla \phi|^{p-2} 
\biggl\langle\nabla \left( \frac{|u|^p}{|\phi|^{p-2} \phi} \right) , \nabla \phi \biggr\rangle 
\end{align*}
and
\begin{align*}
C_p(\xi,\eta)=R_p(\xi,\eta)\geq0,
\end{align*}
where $C_p(\cdot,\cdot)$ is given in (\ref{cp formula}) and
\begin{align*}
\xi := \nabla u, \quad \eta := \nabla u - \frac{\nabla \phi}{\phi} u.
\end{align*}
\end{thm}

We refer to the crucial work of Allegretto and Huang \cite{allegretto1998picones}, where the Picone's identity for real-valued functions made its first appearance.

\section{Main results}\label{main results}

In this section, we state the stability of the $L^{p}$-Poincar\'e inequality. First, we show results concerning the usual Euclidean $p$-Laplacian, then for the Gaussian $p$-Laplacian. We then apply the stability inequality to establish known and new fundamental gap estimates.

\begin{thm}[\cite{apseit2025sharp}]\label{thm apseit}
Let $1<p<\infty$, $\Omega\subset \mathbb{R}^n$ be an open bounded set. Then, for all complex-valued $u\in C^{\infty}_{0}(\Omega)$, we have
\begin{align}\label{general poincare}
\int_{\Omega}|\nabla u|^{p}dx-\lambda_1(p, \Omega)\int_{\Omega}|u|^{p}dx=\int_{\Omega}C_{p}\left(\nabla u,u_{1}\nabla\left(\frac{u}{u_{1}}\right)\right)dx.
\end{align}
Moreover, the remainder term vanishes if and only if $u=cu_{1}$ for all $c\in \mathbb{C}$.
\end{thm}

\begin{rem}
The identity, as with all such identities, thanks to Lemmata \ref{lem1}, \ref{lem2} and \ref{lem3} allows us to derive all possible optimizers of the $L^{p}$-Poincar\'e inequality for any $p\in \left(1,\infty\right)$. Such set, we will denote as $E_{Poin}:=\{cu_{1}: c\in \mathbb{C}\}$. If we restrict the $L^{p}$-Poincar\'e inequality to real-valued functions, then the set of all optimizers becomes $E_{Poin}=\{cu_{1}: c\in \mathbb{R}\}$.
\end{rem}

The identity (\ref{general poincare}) is the starting point in proving the stability result. In combination with Theorem \ref{thm saka} and Theorem \ref{thm log-convex stab}, we finally obtain the stability result of the $L^{p}$-Poincar\'e inequality: 

\begin{thm}\label{thm stability poincare}
Let $p, n\geq 2$, $\Omega\subset\mathbb{R}^n$ be an open bounded convex set having diameter $\mathrm{diam}(\Omega)$ with smooth boundary $\partial \Omega$. Then, for every real-valued $u\in C_{0}^{\infty}(\Omega)$, we have
\begin{align}\label{stability}
\int_{\Omega}|\nabla u|^{p}dx-\lambda_{1}(p,\Omega)\int_{\Omega}|u|^{p}dx \geq c_{1}(p)\left(\frac{\pi_{p}}{\mathrm{diam}(\Omega)}\right)^{p}d(u, E_{Poin})^{p},
\end{align}
where $c_{1}(p)$ is an explicit constant given in (\ref{cp sharp formula}). Here, $E_{Poin}=\{cu_{1}: c\in \mathbb{R}\}$ is the manifold of optimizers (or eigenspace) of (\ref{poincare}),
\begin{align*}
\pi_{p}=2 \int_{0}^{+\infty} \frac{1}{1+\frac{1}{p-1} s^{p}} d s=2 \pi \frac{(p-1)^{1 / p}}{p(\sin (\pi / p))}
\end{align*}
and $d(u, E_{Poin})$ is the $L^{p}$ distance to the manifold of optimizers, i.e.
\begin{align*}
d(u, E_{Poin})^{p} := \inf_{c\in\mathbb{R}}\left\{\norm{u-cu_{1}}^{p}_{L^{p}(\Omega)}\right\}.
\end{align*}
\end{thm}

When $p=2$, in (\ref{stability}), interestingly the constants are $c_{1}(2)=1$ and $\pi_{2}=\pi$, which implies:

\begin{cor}\label{cor 1}
Let $n\geq2$ and $\Omega\subset \mathbb{R}^n$ be an open bounded convex set having diameter $\mathrm{diam}(\Omega)$ with smooth boundary $\partial \Omega$. Then, for every real-valued $u\in C_{0}^{\infty}(\Omega)$, we have
\begin{align}\label{l2 stab}
\int_{\Omega}|\nabla u|^{2}dx-\lambda_{1}(2,\Omega)\int_{\Omega}|u|^{2}dx \geq \left(\frac{\pi}{\mathrm{diam}(\Omega)}\right)^{2}d(u, E_{Poin})^{2}.
\end{align}
\end{cor}

\begin{rem}
The density argument can be applied to extend the inequalities (\ref{stability}) and (\ref{l2 stab}) to $u\in W_{0}^{1,p}(\Omega)$, which we will use for obtaining the fundamental gap of the $p$-Laplacian.
\end{rem}

The inequality (\ref{stability}) actually allows us to obtain a fundamental gap of the Dirichlet $p$-Laplacian on bounded convex domains with smooth boundary:

\begin{cor}\label{cor 2}
Let $p, n$ and $\Omega$ be from Theorem \ref{thm stability poincare} with normalized eigenfunctions, i.e. $\norm{u_{1}}_{L^{p}(\Omega)}=\norm{u_{2}}_{L^{p}(\Omega)}=1$. Then, for $\lambda_{1}(p,\Omega)$ and $\lambda_{2}(p,\Omega)$ defined as in (\ref{eigenvalue first})-(\ref{eigenvalue second}), we have
\begin{align}\label{gap p}
\lambda_{2}(p,\Omega)-\lambda_{1}(p,\Omega) \geq c_{1}(p)\left(\frac{\pi_{p}}{\mathrm{diam}(\Omega)}\right)^{p}C(p, \Omega, u_{1}, u_{2}),
\end{align}
where
\begin{align*}
C(p, \Omega, u_{1}, u_{2}):=\inf_{c\in \mathbb{R}}\int_{\Omega}|u_{2}-cu_{1}|^{p}dx.
\end{align*}
\end{cor}

\begin{rem}\label{simple rem}
One can use the lower bound of $c_{1}(p)$ from (\ref{lower bound cp}) to obtain a simpler expression of the stability constant:
\begin{align}\label{explicit stability rem}
\int_{\Omega}|\nabla u|^{p}dx-\lambda_{1}(p,\Omega)\int_{\Omega}|u|^{p}dx
\geq
\frac{1}{2^{p-2}}
\left(\frac{\pi_{p}}{\mathrm{diam}(\Omega)}\right)^{p}
d(u,E_{Poin})^{p}.
\end{align}
Thus, from (\ref{gap p}), we get
\begin{align}
\label{explicit gap rem}
\lambda_{2}(p,\Omega)-\lambda_{1}(p,\Omega)
\geq
\frac{1}{2^{p-2}}
\left(\frac{\pi_{p}}{\mathrm{diam}(\Omega)}\right)^{p}
C(p,\Omega,u_{1},u_{2}).
\end{align}
Here, we also note that after our preprint appeared on arXiv, Chen and Hauer, among other results, extended the corresponding gap bounds to convex potentials and showed that the eigenfunction-dependent factor can be removed; see \cite[Lemma 6.3]{chen2026fundamental} and \cite[Theorem 1.7(i)]{chen2026fundamental}. Indeed, combining the mountain-pass characterization \cite[Formula (2.4)]{chen2026fundamental} with \cite[Lemma 6.3]{chen2026fundamental} and our stability estimate in Theorem \ref{thm stability poincare}, one obtains
\begin{align*}
\lambda_{2}(p,\Omega)-\lambda_{1}(p,\Omega)
\geq
c_{1}(p)
\left(\frac{\pi_{p}}{\mathrm{diam}(\Omega)}\right)^{p}.
\end{align*}
Using the lower bound (\ref{lower bound cp}), this further yields
\begin{align*}
\lambda_{2}(p,\Omega)-\lambda_{1}(p,\Omega)
\geq
\frac{1}{2^{p-2}}
\left(\frac{\pi_{p}}{\mathrm{diam}(\Omega)}\right)^{p}.
\end{align*}
\end{rem}

\begin{rem}
In the aforementioned survey \cite[Page 7]{dai2018fundamental}, Dai, Seto and Wei comment that spectral gaps for the Dirichlet $p$-Laplacian operator are \say{still unknown} and we believe it is the first such result in this direction.
\end{rem}

When $p=2$, in the inequality (\ref{gap p}), we have that $C(2,\Omega,u_{1},u_{2})=1$ and $\pi_{2}=\pi$, which gives us the fundamental gap of the Laplacian in the form of Yu, Zhong and Smits \cite{yu1986lower, smits1996spectral}.

\begin{cor}\label{cor 3}
Let $n$ and $\Omega$ be from Theorem \ref{thm stability poincare} with normalized eigenfunctions. Then, for $\lambda_{1}(2,\Omega)$ and $\lambda_{2}(2,\Omega)$ defined as in (\ref{eigenvalue first})-(\ref{eigenvalue second}), we have
\begin{align*}
\lambda_{2}(2, \Omega)-\lambda_{1}(2, \Omega)\geq \frac{\pi^{2}}{\mathrm{diam}(\Omega)^{2}}.
\end{align*}
\end{cor}

\begin{rem}
We remark that, as pointed out to us by V. Bobkov, from the inequality (\ref{takac ineq}), the stability of the $L^{p}$-Poincar\'e inequality in the form of Bianchi and Egnell can be obtained as follows:

\begin{multline}\label{takac egnell ineq}
\int_{\Omega} |\nabla u|^p dx - \lambda_1 \int_{\Omega} |u|^p dx \geq c\int_{\Omega} |\nabla u^{\perp}|^p dx = c\int_{\Omega}\left|\nabla\left(u-u^{\parallel}u_{1}\right)\right|^{p}dx \\\geq c\inf_{U\in E_{Poin}}\int_{\Omega}\left|\nabla\left(u-U\right)\right|^{p}dx,
\end{multline}
where $E_{Poin}$ is the eigenspace. However, the constant, in (\ref{takac egnell ineq}), is not explicit. Here, in this paper, we will demonstrate a different approach, which allows us to obtain a stability result with an explicit constant. 
\end{rem}

Next, we show some corresponding results for the Gaussian probability measure.

\begin{thm}\label{thm1}
Let $1<p<\infty$, $\Omega\subset \mathbb{R}^n$ be an open bounded set. Then, for all complex-valued $u\in C^{\infty}_{0}(\Omega)$, we have
\begin{align}\label{general poincare gauss}
\int_{\Omega}|\nabla u|^{p}d\gamma-\lambda_1(p, \Omega,\gamma)\int_{\Omega}|u|^{p}d\gamma=\int_{\Omega}C_{p}\left(\nabla u,u^{\gamma}_{1}\nabla\left(\frac{u}{u^{\gamma}_{1}}\right)\right)d\gamma.
\end{align}
Moreover, the remainder term vanishes if and only if $u=cu^{\gamma}_{1}$ for all $c\in \mathbb{C}$.
\end{thm}

\begin{rem}
The identity (\ref{general poincare gauss}) implies that all optimizers of the inequality
\begin{align*}
\int_{\Omega}|\nabla u|^{p}d\gamma \geq \lambda_{1}(p,\Omega, \gamma)\int_{\Omega}|u|^{p}d\gamma
\end{align*}
are in the set that we denote $E^{\gamma}_{Poin}:=\{cu^{\gamma}_{1}:c\in \mathbb{C}\}$.
\end{rem}

By the same argument as in the proof of Theorem \ref{thm stability poincare}, we obtain the stability of the $L^{p}$-Poincar\'e inequality for Gaussian probability measures.

\begin{thm}\label{thm stability gauss}
Let $p, n\geq 2$ and $\Omega\subset\mathbb{R}^n$ be an open bounded convex set having the diameter $\mathrm{diam}(\Omega)$. Then, for every real-valued $u\in C_{0}^{\infty}(\Omega)$, we have
\begin{align}\label{stability gauss}
\int_{\Omega}|\nabla u|^{p}d\gamma - \lambda_{1}(p,\Omega,\gamma)\int_{\Omega}|u|^{p}d\gamma \geq c_{1}(p)\left(\frac{\pi_{p}}{\mathrm{diam}(\Omega)}\right)^{p}d(u,E_{Poin}^{\gamma})^{p},
\end{align}
where $c_{1}(p)$ is defined by (\ref{cp sharp formula}),
\begin{align*}
\pi_{p}=2 \int_{0}^{+\infty} \frac{1}{1+\frac{1}{p-1} s^{p}} d s=2 \pi \frac{(p-1)^{1 / p}}{p(\sin (\pi / p))}
\end{align*}
and $d(u, E^{\gamma}_{Poin})$ is the $L^{p}$ distance to the manifold of optimizers, i.e.
\begin{align*}
d(u, E^{\gamma}_{Poin})^{p} := \inf_{c\in\mathbb{R}}\left\{\norm{u-cu^{\gamma}_{1}}^{p}_{L^{p}(\Omega, \gamma)}\right\}.
\end{align*}
\end{thm}

\begin{rem}
As we noted in Section \ref{sec prem}, Theorem \ref{thm convex gauss} of Colesanti, Qin and Salani does not require the boundary of the domain $\partial \Omega$ to be smooth. Thus, in the case of Gaussian probability measure, Theorem \ref{thm stability gauss} does not assume $\partial \Omega$ to be smooth.
\end{rem}

In the same spirit as in Corollary \ref{cor 1}, for $p=2$ in (\ref{stability gauss}), we have

\begin{cor}\label{cor 1 gauss}
Let $n\geq2$ and $\Omega\subset\mathbb{R}^n$ be an open bounded convex set having the diameter $\mathrm{diam}(\Omega)$. Then, for every real-valued $u\in C_{0}^{\infty}(\Omega)$, we have
\begin{align*}
\int_{\Omega}|\nabla u|^{2}d\gamma-\lambda_{1}(2,\Omega, \gamma)\int_{\Omega}|u|^{2}d\gamma \geq \left(\frac{\pi}{\mathrm{diam}(\Omega)}\right)^{2}d(u, E^{\gamma}_{Poin})^{2}.
\end{align*}
\end{cor}

By density of $C_{0}^{\infty}(\Omega)$ in $W_{0}^{1,2}(\Omega, \gamma)$, Corollary \ref{cor 1 gauss} extends to all $u\in W_{0}^{1,2}(\Omega, \gamma)$. Thus, by the identical argument as for the proof of Corollary \ref{cor 2}, we are able to obtain a spectral gap estimate in Gaussian spaces.

\begin{cor}\label{cor 333}
Let $n$ and $\Omega$ be from Corollary \ref{cor 1 gauss} with $L^{2}(\Omega, \gamma)$-normalized eigenfunctions. Then for $\lambda_{1}(2, \Omega, \gamma)$ and $\lambda_{2}(2, \Omega, \gamma)$ defined as in (\ref{def eigenvalues gauss}), we have
\begin{align}\label{gap gauss}
\lambda_{2}(2, \Omega, \gamma)-\lambda_{1}(2, \Omega, \gamma) \geq \frac{\pi^{2}}{\mathrm{diam}(\Omega)^{2}}.
\end{align}
\end{cor}

\begin{rem}
We note that the optimal spectral gap $\frac{3\pi^{2}}{\mathrm{diam}(\Omega)^{2}}$ in Gaussian spaces has been shown in a very recent paper \cite{sun2025sharp} by Sun and Wang. Although our approach gives a suboptimal constant, it gives a constructive way of generalizing such results for any $p\geq 2$. Indeed, if the second eigenvalue $\lambda_{2}(p, \Omega, \gamma)$ of the Gaussian $p$-Laplacian can be defined similarly as in the Lebesgue setting, then an analogous spectral gap estimate would take the following form:
\begin{align}\label{gap gauss p}
\lambda_{2}(p,\Omega, \gamma)-\lambda_{1}(p,\Omega, \gamma) \geq c_{1}(p)\left(\frac{\pi_{p}}{\mathrm{diam}(\Omega)}\right)^{p}C(p,\Omega,u^{\gamma}_{1},u^{\gamma}_{2}, \gamma),
\end{align}
where
\begin{align*}
C(p,\Omega,u^{\gamma}_{1},u^{\gamma}_{2}, \gamma):=\inf_{c\in \mathbb{R}}\int_{\Omega}|u^{\gamma}_{2}-cu^{\gamma}_{1}|^{p}d\gamma.
\end{align*}
\end{rem}

\begin{rem}
In a similar way, as pointed out in Remark \ref{simple rem}, by using the estimate of $c_{1}(p)\geq \frac{1}{2^{p-2}}$ in (\ref{stability gauss}), we also have  
\begin{align}\label{explicit stability gauss rem}
\int_{\Omega}|\nabla u|^{p}d\gamma - \lambda_{1}(p,\Omega,\gamma)\int_{\Omega}|u|^{p}d\gamma \geq \frac{1}{2^{p-2}}\left(\frac{\pi_{p}}{\mathrm{diam}(\Omega)}\right)^{p}d(u,E_{Poin}^{\gamma})^{p}
\end{align}
and, therefore, from (\ref{gap gauss p}), we have
\begin{align}\label{explicit gap gauss rem}
\lambda_{2}(p,\Omega, \gamma)-\lambda_{1}(p,\Omega, \gamma) \geq \frac{1}{2^{p-2}}\left(\frac{\pi_{p}}{\mathrm{diam}(\Omega)}\right)^{p}C(p,\Omega,u^{\gamma}_{1},u^{\gamma}_{2}, \gamma).
\end{align}
\end{rem}

\section{Proof of Proposition \ref{cp asymp}}\label{sec proof prop}
\begin{proof}[Proof of Proposition \ref{cp asymp}]
Let us recall that
\begin{align}\label{cp proof formula}
c_{1}(p)=(p-1)(r_{0}+1)^{2-p},
\end{align}
where $r_{0}$ is the root of the function
\begin{align*}
f(r)=r^{p-1}-(p-1)r-(p-2).
\end{align*}
When $p=2$, we have $c_{1}(2)=1$. When $p>2$, we first take the derivative 
\begin{align*}
f'(r)=(p-1)(r^{p-2}-1).
\end{align*}
For any $p>2$, we see that $f'(r)=0$ at $r=1$. In addition, for any $p>2$ and $r>1$, one has $f'(r)>0$. Also note that $f(1)=4-2p$ starts at a negative value and becomes positive for large $r$. By continuity of $f(r)$, there exists at least one root $r_{0}>1$. Now since $f(r)$ is strictly increasing on $r\in (1,\infty)$, by the Intermediate Value Theorem, the root is unique and satisfies $r_{0}>1$.

Since $r_{0}> 1$, we have $r_{0}+1> 2$ and consequently $(r_{0}+1)^{p-2}> 2^{p-2}$. Therefore, from (\ref{cp proof formula}), we obtain that for $p>2$:
\begin{align*}
c_{1}(p)< \frac{p-1}{2^{p-2}}.
\end{align*}
Combining $p=2$ and $p>2$ cases, we get that
\begin{align*}
c_{1}(p)\leq \frac{p-1}{2^{p-2}}.
\end{align*}

To get a lower bound, we construct a new function $g(t)=t^{q}$, which is convex for any $q\geq 1$ and $t>0$. Since $g(t)$ is convex, it is also midpoint convex:
\begin{align*}
g\left(\frac{r_{0}+1}{2}\right)\leq \frac{g(r_{0})+g(1)}{2},
\end{align*}
which is equivalent to
\begin{align*}
\left(\frac{r_{0}+1}{2}\right)^{q}\leq\frac{r^{q}_{0}+1}{2}.
\end{align*}
Applying the inequality for $q=p-1$ ($p\geq 2$) and multiplying both sides by $2^{p-1}$, gives us
\begin{align*}
(r_{0}+1)^{p-1}\leq(r^{p-1}_{0}+1)2^{p-2}.
\end{align*}
Since $r^{p-1}_0+1=(r_{0}+1)(p-1)$, we get
\begin{align*}
(r_{0}+1)^{p-1}\leq 2^{p-2}(r_{0}+1)(p-1).
\end{align*}
Dividing both sides by $r_0+1$, we obtain
\begin{align*}
(r_{0}+1)^{p-2}\leq 2^{p-2}(p-1).
\end{align*}
Therefore, from (\ref{cp proof formula}), for $p\geq 2$, we have 
\begin{align*}
c_{1}(p)\geq \frac{1}{2^{p-2}}.
\end{align*}
By the Squeeze Theorem, it follows that $c_{1}(p) \to 0$ as $p\to \infty$.
\end{proof}

\section{Proofs of Theorem \ref{thm stability poincare} and Corollary \ref{cor 2}: Lebesgue measure}\label{sec proofs lebesgue}

\begin{proof}[\textbf{Proof of Theorem \ref{thm stability poincare}}]
Let $\Omega$ be a bounded convex domain of $\mathbb{R}^n$ ($n\geq 2$). Due to Theorem \ref{thm apseit}, we have that
\begin{align*}
\int_{\Omega}|\nabla u|^{p}dx-\lambda_1(p, \Omega)\int_{\Omega}|u|^{p}dx=\int_{\Omega}C_{p}\left(\nabla u,u_{1}\nabla\left(\frac{u}{u_{1}}\right)\right)dx
\end{align*}
for all $u\in C_{0}^{\infty}(\Omega)$. Using Lemma \ref{lem1}, we have that for $p\geq 2$,
\begin{align}\label{cp step}
\int_{\Omega}|\nabla u|^{p}dx-\lambda_1(p, \Omega)\int_{\Omega}|u|^{p}dx\geq c_{1}(p)\int_{\Omega}\left|\nabla\left(\frac{u}{u_{1}}\right)\right|^{p}|u_{1}|^{p}dx.
\end{align}
Let us define a function
\begin{align*}
f:=\frac{u}{u_{1}}.
\end{align*}
The right hand side of (\ref{cp step}) becomes
\begin{align*}
c_{1}(p)\int_{\Omega}\left|\nabla f\right|^{p}|u_{1}|^{p}dx.
\end{align*}
Now, at this step, we would like to use Theorem \ref{thm log-convex stab}. To do so, we need to show that $f$ is Lipschitz continuous. Since $u\in C_{0}^{\infty}(\Omega)$ and $u_{1}>0$ in $\Omega$, the function $u_{1}$ is bounded away from zero on a neighborhood of $\text{supp}\ u$. Hence, $f$ is well-defined and Lipschitz (see, also, Section \ref{sec prem}). For convenience, let us denote $\omega=|u_{1}|^{p}$. Since Theorem \ref{thm log-convex stab} operates on log-concave measures, it remains to verify that $|u_{1}|^{p}$ is log-concave. By Theorem \ref{thm saka}, the logarithm of the first eigenfunction, $\log u_{1}$, is concave. Observing that
\begin{align*}
\log |u_{1}|^{p} = p\log u_{1},
\end{align*}
we conclude that $|u_{1}|^{p}$ is also log-concave. The last condition which needs to be satisfied is
\begin{align}\label{condition}
\int_{\Omega}|f|^{p-2}f\omega dx = 0.
\end{align}
However, this is not true in general. That is why we introduce a new function
\begin{align*}
g(t):=\int_{\Omega}|f-t|^{p-2}(f-t)\omega dx.
\end{align*}
By the Dominated Convergence Theorem, we have the continuity of $g(t)$ with respect to $t$. Since $f$ is bounded, we have that
\begin{align*}
m\leq f \leq M,
\end{align*}
where $m=\inf f$ and $M = \sup f$. Let us choose any $t<m$. Then,
\begin{align*}
g(t)\geq(m-t)^{p-1}\int_{\Omega}\omega dx > 0,
\end{align*}
which implies that $g$ attains a positive value. If $t>M$, then
\begin{align*}
g(t)\leq -(t-M)^{p-1}\int_{\Omega}\omega dx < 0,
\end{align*}
implying that $g$ also attains a negative value. Thus, since $g$ attains both negative and positive values and is continuous, by the Intermediate Value Theorem, there exists a constant $t_{0}$ such that $g(t_0)=0$. This means
\begin{align*}
\int_{\Omega}|f-t_{0}|^{p-2}(f-t_{0})\omega dx = 0.
\end{align*}
Let us define
\begin{align*}
\tilde{f}:=f-t_{0}.
\end{align*}
Since $\tilde{f}$ is Lipschitz continuous and it satisfies the condition (\ref{condition}), we can finally apply Theorem \ref{thm log-convex stab}:
\begin{align*}
c_{1}(p)\int_{\Omega}|\nabla \tilde{f}|^{p}|u_{1}|^{p}dx \geq c_{1}(p)C_{\Omega, p, \omega}\inf_{t\in\mathbb{R}}\int_{\Omega}|\tilde{f}-t|^{p}|u_{1}|^{p}dx.
\end{align*}
Making use of the identities, $\nabla \tilde{f} = \nabla f$ and $\tilde{f}-t:=f-c$, we get
\begin{align*}
c_{1}(p)\int_{\Omega}|\nabla f|^{p}|u_{1}|^{p}dx \geq c_{1}(p)C_{\Omega, p, \omega}\inf_{c\in\mathbb{R}}\int_{\Omega}|f-c|^{p}|u_{1}|^{p}dx.
\end{align*}
Substituting back, we obtain
\begin{align*}
c_{1}(p)\int_{\Omega}\left|\nabla\left(\frac{u}{u_{1}}\right)\right|^{p}|u_{1}|^{p}dx& \geq c_{1}(p)C_{\Omega, p, \omega}\inf_{c\in\mathbb{R}}\int_{\Omega}\left|\frac{u}{u_{1}}-c\right|^{p}|u_{1}|^{p}dx \nonumber
\\&= c_{1}(p)C_{\Omega, p, \omega}\inf_{c\in\mathbb{R}}\int_{\Omega}\left|u-cu_{1}\right|^{p}dx.
\end{align*}
In combination with (\ref{cp step}), we get
\begin{align*}
\int_{\Omega}|\nabla u|^{p}dx-\lambda_1(p, \Omega)\int_{\Omega}|u|^{p}dx\geq c_{1}(p)C_{\Omega,p,\omega}\inf_{c\in\mathbb{R}}\int_{\Omega}\left|u-cu_{1}\right|^{p}dx.
\end{align*}
Finally, applying the bound on the constant $C_{\Omega,p,\omega}$ from Theorem \ref{thm log-convex stab}, we derive
\begin{align*}
\int_{\Omega}|\nabla u|^{p}dx-\lambda_1(p, \Omega)\int_{\Omega}|u|^{p}dx\geq c_{1}(p)\left(\frac{\pi_{p}}{\mathrm{diam}(\Omega)}\right)^{p}\inf_{c\in\mathbb{R}}\int_{\Omega}\left|u-cu_{1}\right|^{p}dx,
\end{align*}
which completes the proof.
\end{proof}

\begin{proof}[\textbf{Proof of Corollary \ref{cor 2}}]
By Theorem \ref{thm stability poincare}, we have
\begin{align}\label{proof spec 1}
\int_{\Omega}|\nabla u|^{p}dx-\lambda_{1}(p,\Omega)\int_{\Omega}|u|^{p}dx \geq c_{1}(p)\left(\frac{\pi_{p}}{\mathrm{diam}(\Omega)}\right)^{p}\inf_{c\in\mathbb{R}}\int_{\Omega}\left|u-cu_{1}\right|^{p}dx.
\end{align}
We then extend the inequality (\ref{proof spec 1}) to $u\in W_{0}^{1,p}(\Omega)$ using the density argument. Now let $u=u_{2}$ be the second normalized eigenfunction of Dirichlet $p$-Laplacian, i.e. $\norm{u_{2}}_{L^{p}(\Omega)}=1$ and
\begin{align*}
\int_{\Omega}|\nabla u_{2}|^{p}dx=\lambda_{2}(p,\Omega)\int_{\Omega}|u_{2}|^{p}dx=\lambda_{2}.
\end{align*}
Then, in (\ref{proof spec 1}), we get
\begin{align*}
\lambda_{2}(p,\Omega)-\lambda_{1}(p,\Omega) \geq c_{1}(p)\left(\frac{\pi_{p}}{\mathrm{diam}(\Omega)}\right)^{p}\inf_{c\in\mathbb{R}}\int_{\Omega}\left|u_{2}-cu_{1}\right|^{p}dx,
\end{align*}
which completes the proof.
\end{proof}

\section{Proofs of Theorems \ref{thm1} and \ref{thm stability gauss}: Gaussian probability measure}\label{sec proofs gaussian}

\begin{proof}[\textbf{Proof of Theorem \ref{thm1}}]
From Theorem \ref{thm picone}, we have
\begin{align*}
C_{p}(\xi,\eta)=R_{p}(\xi,\eta)=|\nabla u|^p - |\nabla\phi|^{p-2} \biggl\langle\nabla \left( \frac{|u|^p}{|\phi|^{p-2} \phi} \right) , \nabla \phi \biggr \rangle.
\end{align*}
Multiplying both sides by $\left(2\pi\right)^{-\frac{n}{2}}e^{-\frac{|x|^{2}}{2}}$ and integrating over $\Omega$ give
\begin{align}\label{integrate}
\int_{\Omega}C_{p}(\xi,\eta)d\gamma=\int_{\Omega}|\nabla u|^pd\gamma - \int_{\Omega}|\nabla\phi|^{p-2} \biggl \langle\nabla \left( \frac{|u|^p}{|\phi|^{p-2} \phi} \right) , \nabla \phi \biggr \rangle d\gamma .
\end{align}
To proceed next, we define the vector field
\begin{align*}
X:=\frac{|u|^{p}}{|\phi|^{p-2}\phi}|\nabla \phi|^{p-2}\nabla \phi.
\end{align*}
Applying the product rule for the Gaussian divergence with $V=|\nabla \phi|^{p-2}\nabla \phi$ and $f=\frac{|u|^{p}}{|\phi|^{p-2}\phi}$, we get
\begin{align}\label{div X}
\text{div}_{\gamma}X=\biggl\langle \nabla\left(\frac{|u|^{p}}{|\phi|^{p-2} \phi}\right),  \nabla \phi \biggr \rangle |\nabla \phi|^{p-2}+\frac{|u|^{p}}{|\phi|^{p-2} \phi} \operatorname{div}_{\gamma}\left(|\nabla \phi|^{p-2} \nabla \phi\right).
\end{align}
Note that by definition $\text{div}_{\gamma}\left(|\nabla \phi|^{p-2} \nabla \phi\right)=\Delta_{p,\gamma}\phi$. Rewriting (\ref{div X}):
\begin{align}\label{sub}
|\nabla \phi|^{p-2}\biggl\langle \nabla\left(\frac{|u|^{p}}{|\phi|^{p-2} \phi}\right),  \nabla \phi \biggr \rangle =\text{div}_{\gamma}X-\frac{|u|^{p}}{|\phi|^{p-2} \phi} \Delta_{p,\gamma}\phi.
\end{align}
Substituting (\ref{sub}) into (\ref{integrate}):
\begin{align*}
\int_{\Omega}C_{p}\left(\xi,\eta\right)d\gamma = \int_{\Omega}|\nabla u|^{p}d\gamma - \int_{\Omega}\text{div}_{\gamma}X d\gamma + \int_{\Omega}\frac{|u|^{p}}{|\phi|^{p-2}\phi}\Delta_{p,\gamma}\phi d\gamma.
\end{align*}
After integrating by parts and substituting back $\xi$ and $\eta$, we get 
\begin{align*}
\int_{\Omega}|\nabla u|^{p} d\gamma + \int_{\Omega}\frac{|u|^{p}}{|\phi|^{p-2}\phi}\Delta_{p,\gamma}\phi d\gamma=\int_{\Omega}C_{p}\left(\nabla u, \phi\nabla\left(\frac{u}{\phi}\right)\right)d\gamma.
\end{align*}
Finally, setting $\phi$ to be the first eigenfunction $u^{\gamma}_{1}$ of the Gaussian $p$-Laplacian, we obtain
\begin{align*}
\int_{\Omega}|\nabla u|^{p}d\gamma - \lambda_{1}(p,\Omega,\gamma)\int_{\Omega}|u|^{p}d\gamma = \int_{\Omega}C_{p}\left(\nabla u, u_{1}^{\gamma}\nabla\left(\frac{u}{u_{1}^{\gamma}}\right)\right)d\gamma.
\end{align*}
The proof is complete. 
\end{proof}

\begin{proof}[\textbf{Proof of Theorem \ref{thm stability gauss}}]
Set $\Omega$ to be a bounded convex domain of $\mathbb{R}^n$ ($n \geq 2$). By Theorem \ref{thm1}, we have
\begin{align*}
\int_{\Omega}|\nabla u|^{p}d\gamma-\lambda_{1}(p,\Omega, \gamma)\int_{\Omega}|u|^{p}d\gamma=\int_{\Omega}C_{p}\left(\nabla u, u^{\gamma}_{1}\nabla\left(\frac{u}{u^{\gamma}_{1}}\right)\right)d\gamma
\end{align*}
for all real-valued $u\in C_{0}^{\infty}(\Omega)$. Lemma \ref{lem1} implies that, for $p\geq 2$, we have
\begin{align}\label{gauss pre last}
\int_{\Omega}|\nabla u|^{p}d\gamma-\lambda_{1}(p,\Omega, \gamma)\int_{\Omega}|u|^{p}d\gamma \geq c_{1}(p)\int_{\Omega}\left|\nabla\left(\frac{u}{u_{1}^{\gamma}}\right)\right|^{p}|u^{\gamma}_{1}|^{p}d\gamma.
\end{align}
Now we define
\begin{align*}
f:=\frac{u}{u^{\gamma}_{1}}.
\end{align*}
The regularity properties of $u$ and $u^{\gamma}_{1}$ imply that $f$ is Lipschitz continuous (see, Proposition \ref{prop regularity gaussian}). Since $|u^{\gamma}_{1}|^{p}$ is log-concave by Theorem \ref{thm convex gauss}, we have that 
\begin{align*}
|u^{\gamma}_{1}|^{p}\left(2\pi\right)^{-\frac{n}{2}}e^{-\frac{|x|^{2}}{2}}
\end{align*} 
is also log-concave. Let us denote $\omega=|u^{\gamma}_{1}|^{p}\left(2\pi\right)^{-\frac{n}{2}}e^{-\frac{|x|^{2}}{2}}$. Then, by an analogous argument from the proof of Theorem \ref{thm stability poincare}, we then eventually get from Theorem \ref{thm log-convex stab}:
\begin{align*}
c_{1}(p)\int_{\Omega}|\nabla f|^{p}|u^{\gamma}_{1}|^{p}d\gamma \geq c_{1}(p) C_{\Omega, p, \omega}\inf_{c\in \mathbb{R}}\int_{\Omega}|f-c|^{p}|u^{\gamma}_{1}|^{p}d\gamma
\end{align*}
for $C_{\Omega,p,\omega}>0$. After the change of variables and applying it in combination with the estimate (\ref{gauss pre last}), we get
\begin{align*}
\int_{\Omega}|\nabla u|^{p}d\gamma - \lambda_{1}(p,\Omega, \gamma)\int_{\Omega}|u|^{p}d\gamma \geq c_{1}(p)C_{\Omega, p, \omega}\inf_{c\in \mathbb{R}}\int_{\Omega}|u-cu^{\gamma}_{1}|^{p}d\gamma.
\end{align*}
Applying the lower bound on the constant $C_{\Omega, p, \omega}$ from Theorem \ref{thm log-convex stab}, we derive 
\begin{align*}
\int_{\Omega}|\nabla u|^{p}d\gamma-\lambda_1(p, \Omega,\gamma)\int_{\Omega}|u|^{p}d\gamma\geq c_{1}(p)\left(\frac{\pi_{p}}{\mathrm{diam}(\Omega)}\right)^{p}\inf_{c\in\mathbb{R}}\int_{\Omega}\left|u-cu^{\gamma}_{1}\right|^{p}d\gamma,
\end{align*}
which completes the proof.
\end{proof}

{\bf Acknowledgments.} We would like to thank Vladimir Bobkov for valuable discussions and comments as well as for pointing out an important reference \cite{fleckinger2002improved} in the field. 



\bibliographystyle{alpha}
\bibliography{citation}

\end{document}